\documentclass{amsart}

\usepackage[all]{xy}

\chardef\bslash=`\\ % p. 424, TeXbook

\newcommand{\Sing}{\operatorname{Sing}}

\newtheorem{theorem}{Theorem}[section] % 1st argument is your name for it
\newtheorem{lemma}[theorem]{Lemma}     % 2nd argument is what is printed
\newtheorem{corollary}[theorem]{Corollary}

\newtheorem*{hausdorff-limit}{Theorem E}

\newcommand{\eval}[2][\right]{\relax
  \ifx#1\right\relax \left.\fi#2#1\rvert}

\title[Homoclinic classes for sectional-hyperbolic sets]{Homoclinic classes for sectional-hyperbolic sets}
\author{A. Arbieto, C.A. Morales, A.M. Lopez B.}

\address{Instituto de Matem\'atica, Universidade Federal do Rio de Janeiro, P. O. Box 68530, 21945-970 Rio
de Janeiro, Brazil.}
\email{arbieto@im.ufrj.br, barragan@im.ufrj.br, morales@impa.br}

\thanks{Partially supported by CNPq, FAPERJ and PRONEX/DS from Brazil.}

\subjclass[2010]{Primary: 37D20; Secondary: 37C70}
\keywords{Sectional-hyperbolic set, Homoclinic class, Flow.}

\begin{document}

\begin{abstract}
We prove that every sectional-hyperbolic Lyapunov stable set contains a nontrivial homoclinic class.
\end{abstract}

%\part{Use this type of header for very long papers only}
% use lowercase except for proper names
\maketitle
\section{Introduction} % use lowercase except for proper names
\label{intro}

\noindent
A well-known problem in dynamics is to determinate when a given system has periodic or homoclinic orbits.
This problem is completely solved for hyperbolic sets, namely,
every nontrivial isolated hyperbolic set contains homoclinic (and hence infinitely many periodic) orbits.
It is natural to extend this solution beyond
hyperbolicity.
For instance we can consider the {\em singular-hyperbolic sets},
introduced in \cite{mpp} to put together both hyperbolic systems and certain robustly transitive sets with singularities
in dimension three like the {\em geometric Lorenz attractors}
\cite{abs},\cite{gw}.
It is then tempting to say that every nontrivial isolated sectional-hyperbolic set contains homoclinic orbits,
but this is not true in general \cite{m}.
However,
Bautista and the second author proved that if a
singular-hyperbolic set in dimension three is attracting, then it must contain
a periodic orbit \cite{bm}.
This was obtained in parallel with the claim by Arroyo and
Pujals \cite{apu} that 
every singular-hyperbolic attractor in dimension three is a homoclinic class (see also \cite{ap}).
Afterward Nakai \cite{n} extended \cite{bm}
from attracting to Lyapunov stable sets
and Reis \cite{r} gave generic conditions under which
a singular-hyperbolic attracting set in dimension three exhibits infinitely many periodic orbits.
In higher dimensions, Metzger and the second author \cite{memo}
introduced the notion of {\em sectional-hyperbolic sets}
which reduces to singular-hyperbolicity in dimension three.
In this context,
the third author \cite{l} was able to extend
the existence of periodic orbits to all sectional-hyperbolic attracting sets. In this note we go further and prove that 
every sectional-hyperbolic Lyapunov stable set has a nontrivial
homoclinic class. Therefore, all such sets display homoclinic
(and hence infinitely many periodic) orbits.
Let us state our result in a precise way.

By {\em abus de langage}, we call {\em flow} any $C^1$ vector field $X$ with
induced flow $X_t$ of a compact connected manifold
$M$ endowed with a Riemannian structure $\|\cdot\|$.
We say that $\Lambda\subset M$ is
{\em invariant} if $X_t(\Lambda)=\Lambda$ for all $t\in\mathbb{R}$.
An invariant set $\Lambda$ is {\em Lyapunov stable} if
for every neighbourhood $U$ of $\Lambda$ there is a neighbourhood $V\subset U$ of $\Lambda$ such that
$X_t(V)\subset U$ for all $t\geq0$. Similar definition holds for maps.
The set of singularities (i.e. zeroes of $X$) is denoted by $\Sing(X)$.
We say that $\sigma\in\Sing(X)$ is hyperbolic if
the derivative $DX(\sigma)$ has no purely imaginary eigenvalues.
A point $x$ is {\em periodic} if there is a minimal $t=t_x>0$ such that
$X_t(x)=x$. We say that a periodic point $x$ is {\em hyperbolic}
if the eigenvalues of the derivative $DX_{t_x}(x)$ not corresponding to the
flow direction are all different from $1$ in modulus.
In case there are eigenvalues of modulus less and bigger than $1$
we say that the hyperbolic periodic point is a {\em saddle}. 

As is well known \cite{hps}, through any periodic saddle $x$ it passes a pair of invariant manifolds, the so-called strong stable and unstable manifolds
$W^{ss}(x)$ and $W^{uu}(x)$, tangent at $x$ to the eigenspaces corresponding to the eigenvalue of modulus less and bigger than $1$ respectively.
Saturating them with the flow we obtain the stable and unstable manifolds $W^s(x)$ and $W^u(x)$ respectively.
A homoclinic orbit associated to $x$ is the orbit of a point $q$ where these last manifolds meet. If, additionally, $\dim(T_qW^s(x)\cap T_qW^u(x))=1$, then we
say that the homoclinic orbit is transversal.
A {\em homoclinic class} is the closure of the transversal homoclinic orbits of a given periodic saddle. It is {\em nontrivial}
if it does not reduce to a single periodic orbit.

We say that a compact invariant set {\em $\Lambda$
has a dominated splitting with respect to the tangent flow} if there is a continuous splitting
$T_\Lambda M = E\oplus F$ into $DX_t$-invariant subbundles $E,F$ such that $DX_t|_E$ dominates $DX_t|_F$, namely,
there are positive constants $K,\lambda$
satisfying
$$
\|DX_t(p)|_{E_p}\|\cdot\|DX_{-t}(X_t(p))|_{F_{X_t(p)}}\|\leq Ke^{-\lambda t},
\quad\quad\forall p\in \Lambda, t\geq0.
$$
We say that the splitting $T_\Lambda M = E\oplus F$ is a {\em sectional-hyperbolic splitting} if $E$ is {\em contracting}, i.e.,
$$
\|DX_t(p)|_{E_p}\|\leq Ke^{-\lambda t},
\quad\quad\forall p\in \Lambda, t\geq0,
$$
and $F$ is {\em sectional expanding}, i.e., $\dim(F)\geq2$ and
$$
|\det DX_t(p)|_L|\geq Ke^{\lambda t},
$$
for every $p\in \Lambda$, $t\geq0$ and every two-dimensional subspace
$L\subset F_p$.

A compact invariant set is {\em sectional-hyperbolic}
if its singularities are all hyperbolic and if it exhibits
a sectional-hyperbolic splitting.

With these definitions we can state our main result.

\begin{theorem}
\label{ThA}
Every sectional-hyperbolic Lyapunov stable set
contains a nontrivial homoclinic class.
\end{theorem}

The proof relies on recent results concering
hyperbolic ergodic measures for flows \cite{cch}, \cite{ch}, \cite{sgw}.

\section{Proof}

\noindent
We start with some terminology from \cite{ch}.
As it is well-known, the space of probability measures of $M$ endowed
with the weak* topology is metrizable, we denote by $d_*$ the corresponding metric.
We say that a measure $\mu$ is {\em supported on $H\subset M$}
if its support $supp(\mu)$ is contained in $H$.
We denote by $\delta_y$ the Dirac measure supported on $y$.

If $f: M\to M$ is a continuous map,
we say that a Borel probability measure $\mu$ is an {\em invariant measure}
if $\mu(f^{-1}(A))=\mu(A)$ for every Borelian $A$.
For any point $x\in M$ we denote by $p\omega(x)$ the set of all the Borel probabilities measures that are the limits of
the convergent subsequences of the sequence
$$
\frac{1}{n}\displaystyle\sum_{i=0}^{n-1}\delta_{f^i(x)}.
$$
An invariant measure $\mu$ is {\em SRB-like for $f$}
if for all  $\epsilon>0$ the set
of points $x\in M$ such that $d_*(p\omega(x),\mu)<\epsilon$
has positive Lebesgue measure.

Applying Theorem 3.1 in \cite{ch} we obtain the following existence
result.

\begin{lemma}
\label{l.existe}
Every Lyapunov stable set of a continuous map $f$ supports a SBR-like measure.
\end{lemma}

\begin{proof}
Let $\Lambda$ be a Lyapunov stable set of $f$.
Since $\Lambda$ is Lyapunov stable, we can take
a nested sequence $U_i$ of compact neighbourhoods of $\Lambda$
such that $f(U_i)\subset U_i$ and $\bigcap_iU_i=\Lambda$.
By the aforementioned result in \cite{ch} there is a sequence of SRB-like measures
$\mu_i$  for $f|_{U_i}$, $\forall i\in \mathbb{N}$.
By definition, such measures are also SRB-like measures for $f$.
Again by \cite{ch}, any accumulation measure of $\mu_i$
is SBR-like and supported on $\Lambda$. This ends the proof.
\end{proof}

Next we recall some facts about Lyapunov exponents.
Assume that $f$ is a diffeomorphism and
let $\mu$ be an invariant measure.
By Oseledets's Theorem, for every continuous invariant subbundle
$F$ of $T_\Lambda M$ there exits a full measure set $R$
(called regular points) and, for all $x\in R$, a positive integer $k(x)$,
real numbers $\chi_1(x)<\cdots<\chi_{k(x)}(x)$ and a splitting
$F_x=E^1_x\oplus \cdots \oplus E^{k(x)}_x$,
depending measurably on $x\in R$,
such that
$$
\displaystyle\lim_{n\to\pm\infty}\frac{1}{n}\log\|Df^n(x)v^i\|=\chi_i(x),
\quad\quad\forall v^i\in E^i_x\setminus\{0\}, 1\leq i\leq k(x).
$$
The numbers $\chi_i(x)$ (which depends measurably on $x\in R$)
are the so-called {\em Lyapunov exponents} of $\mu$ along $F$.

The following is a corollary of the main result in \cite{cch}.

\begin{lemma}
\label{l1}
Let $\Lambda$ be a Lyapunov stable set of a flow $X$.
If $\Lambda$ has a dominated splitting $T_{\Lambda}M=E\oplus F$
with respect to the tangent flow, and $\mu$ is a SRB-like measure
of the time-$1$ map $X_1$, then
$$
h_\mu(X_1)\geq\displaystyle\int\displaystyle\sum_{i=1}^{\dim(F)}\chi_id\mu,
$$
where $\sum_{i=1}^{\dim(F)}\chi_i$
denotes the sum of the Lyapunov exponents along $F$.
\end{lemma}

The next lemma proves the positivity of the integral of the
sum of the Lyapunov exponents along the central subbundle of any sectional-hyperbolic set.

\begin{lemma}
\label{l2}
Let $\Lambda$ be a compact invariant set of a flow $X$.
If $\Lambda$ has a sectional-hyperbolic splitting
$T_{\Lambda}M=E\oplus F$, and $\mu$ is an invariant measure
of the time-$1$ map $X_1$
supported in $\Lambda$, then
$$
\displaystyle\int\displaystyle\sum_{i=1}^{\dim(F)}\chi_id\mu>0.
$$
\end{lemma}

\begin{proof}
Since
$$
\lim_{n\to\infty}\frac{1}{n}\log|\det DX_n|_F|=\displaystyle\sum_{i=1}^{\dim(F)}\chi_i,
$$
the result follows easily from the sectional expansivity of $F$.
\end{proof}

From this we obtain the following corollary.

\begin{corollary}
\label{coo1}
Every sectional-hyperbolic Lyapunov stable set of
a flow has positive topological entropy.
\end{corollary}

\begin{proof}
Let $\Lambda$ be a sectional-hyperbolic Lyapunov stable set.
By Lemma \ref{l.existe} we can take a SRB-like measure $\mu$ supported on
$\Lambda$
for the restricted time-$1$ map $f=X_1|_{\Lambda}$.
Combining lemmas \ref{l1} and \ref{l2} we obtain $h_\mu(X_1)>0$.
Thus the result follows applying the variational principle to $X_1$.
\end{proof}

The last ingredient is the following lemma
whose proof is contained in that of Theorem 5.6 in \cite{sgw}.
Given a flow $X$ and a compact invariant set $\Lambda$, we say that
{\em $X$ is a star flow on $\Lambda$} if there exists a neighbourhood
$U$ of $\Lambda$, and $\mathcal{U}$ of $X$ in the $C^1$ topology
such that every periodic orbit or singularity contained in $U$ of every flow $Y$
in $\mathcal{U}$ is hyperbolic.

\begin{lemma}
\label{l3}
Let $\Lambda$ be a compact invariant set of a flow $X$.
If $X$ is a star on $\Lambda$, then
the support of any ergodic measure
supported on $\Lambda$ but not on a periodic orbit or singularity
intersects a nontrivial homoclinic class.
\end{lemma}

Now we can prove our main result.

\begin{proof}[Proof of Theorem \ref{ThA}]
Let $\Lambda$ be a sectional-hyperbolic Lyapunov stable set
of a flow $X$.
It is well-known \cite{ap} that $X$ is a star flow on $\Lambda$.
Then, since $\Lambda$ is Lyapunov stable, to prove that there is a nontrivial homoclinic class in $\Lambda$,
it suffices to find by Lemma \ref{l3} an ergodic measure
supported on $\Lambda$ but not on a periodic orbit or singularity.
Since the entropy is positive
by Corollary \ref{coo1}, such a measure can be found by the variational principle for flows. This completes the proof.
\end{proof}

\end{document}